\documentclass[fleqn,12pt,twoside]{article}

\usepackage[headings]{espcrc1}
\readRCS
$Id: espcrc1.tex,v 1.2 2004/02/24 11:22:11 spepping Exp $
\usepackage{graphicx}
\usepackage[figuresright]{rotating}
\newtheorem{theorem}{Theorem}

\newtheorem{claim}{Claim}

\newtheorem{corollary}[theorem]{Corollary}

\newtheorem{lemma}[theorem]{Lemma}

\newtheorem{problem}{Problem}
\newtheorem{proposition}[theorem]{Proposition}

\newtheorem{example}[theorem]{Example}

\newenvironment{proof}[1][Proof.]{\begin{trivlist}
\item[\hskip \labelsep {\bfseries #1}]}{\end{trivlist}}

\newenvironment{acknowledgement}[1][Acknowledgement]{\begin{trivlist}
\item[\hskip \labelsep {\bfseries #1}]}{\end{trivlist}}

\newcommand{\AmS}{{\protect\the\textfont2
  A\kern-.1667em\lower.5ex\hbox{M}\kern-.125emS}}

\usepackage{amsmath}
\usepackage{amsfonts}
\usepackage{amssymb}

\numberwithin{theorem}{section}

\title{Some Results on Cyclic Interval Edge Colorings of Graphs}

\author{Armen S. Asratian\address[MCSD]{Department of Mathematics, Link\"oping University,\\
SE-581 83 Link\"oping, Sweden}\thanks{email: armen.asratian@liu.se},
    Carl Johan Casselgren\addressmark[MCSD]\thanks{email: carl.johan.casselgren@liu.se},
    Petros A. Petrosyan\address{Department of Informatics and Applied Mathematics,\\
Yerevan State University, 0025, Armenia}%
\address{Institute for Informatics and Automation Problems,\\
National Academy of Sciences, 0014, Armenia}%
\thanks{email: petros\_petrosyan@ysu.am, pet\_petros@ipia.sci.am}}

\runtitle{Some Results on Cyclic Interval Edge Colorings of Graphs
}\runauthor{Armen S. Asratian, Carl Johan Casselgren, Petros A.
Petrosyan}

\begin{document}

\maketitle

\begin{abstract}
A proper edge coloring of a graph $G$ with colors $1,2,\dots,t$ is
called a \emph{cyclic interval $t$-coloring} if for each vertex $v$
of $G$ the edges incident to $v$ are colored by consecutive colors,
under the condition that color $1$ is considered as consecutive to
color $t$. We prove that a bipartite graph $G$ of even maximum
degree $\Delta(G)\geq 4$ admits a cyclic interval
$\Delta(G)$-coloring if for every vertex $v$ the degree $d_G(v)$
satisfies either $d_G(v)\geq \Delta(G)-2$ or $d_G(v)\leq 2$. We also
prove that every Eulerian bipartite graph $G$ with maximum degree at
most $8$ has a cyclic interval coloring. Some results are obtained
for $(a,b)$-biregular graphs, that is, bipartite graphs with the
vertices in one part all having degree $a$ and the vertices in the
other part all having degree $b$; it has been conjectured that all
these have cyclic interval colorings. We show that all
$(4,7)$-biregular graphs as well as all $(2r-2,2r)$-biregular
($r\geq 2$) graphs have cyclic interval colorings. Finally, we prove
that all complete multipartite graphs admit cyclic interval
colorings; this proves a conjecture of Petrosyan and Mkhitaryan. \\

\end{abstract}

\section{Introduction}\

 We use \cite{West} for terminology and notation not defined here. 
All graphs considered are
finite, undirected, allow multiple edges and contain no loops, unless otherwise stated. A \emph{simple} graph is a graph
with no loops or multiple edges.
Let $V(G)$ and $E(G)$ denote the sets of vertices and
edges of a graph $G$, respectively.
A proper $t$-edge coloring of a graph $G$ is a
    mapping $\alpha:E(G)\longrightarrow  \{1,\dots,t\}$ such that
    $\alpha(e)\not=\alpha(e')$ for every pair of adjacent
    edges $e$ and $e'$ in $G$. If $e\in E(G)$ and $\alpha(e)=k$ then
    we say that the edge $e$ is {\em colored $k$}.
    We denote by $\Delta(G)$ the maximum degree of
    vertices of a graph $G$, and by $d_G(v)$ the degree of a
    vertex $v$ in $G$. The {\em chromatic index $\chi'(G)$} 
		of a graph $G$
    is the minimum number $t$ for which there exists a proper
    $t$-edge coloring of $G$. By K\H onig's edge coloring
    theorem, $\chi'(G)=\Delta(G)$ for any bipartite graph $G$
    and by Vizing's theorem $\chi'(G)\leq \Delta(G)+1$
    for any simple graph $G$ (see for example \cite{West}).

A proper $t$-edge coloring of a graph $G$ is called 
an \emph{
interval $t$-coloring} if the colors of edges incident to every
vertex $v$ of $G$ form an interval of integers.
 This notion was introduced by Asratian and Kamalian \cite{AsrKam} 
(available in English as \cite{AsrKamJCTB}), motivated by the problem of constructing timetables
 without �gaps� for teachers and classes. Later the theory of interval colorings was developed in
e.g. \cite{AsrCarl,AsrCarlVandWest,CasPetToft,CarlJToft,Giaro,GiaroKub,Hansen,HanLotToft,Kampreprint,KamDiss,KamMirum,PetDM,PetKhachTan,PetKhach,Pyatkin,Seva,YangLi}.
Generally, it is an NP-complete problem to 
determine whether a bipartite graph has an interval 
coloring \cite{Seva}. However some classes of 
graphs have been proved to admit interval colorings. It is known,
for example, that trees, regular and complete bipartite graphs
\cite{AsrKam,Hansen,Kampreprint}, doubly convex bipartite graphs
\cite{KamDiss}, grids \cite{Giaro} and simple outerplanar bipartite
graphs \cite{GiaroKub} have interval colorings. Additionally, all
$(2,b)$-biregular graphs \cite{Hansen,HanLotToft,KamMirum} and
 $(3,6)$-biregular graphs \cite{CarlJToft} admit 
interval colorings, where an
\emph{$(a,b)$-biregular} graph is a bipartite graph
where the vertices in one part all have degree $a$ and the vertices
in the other part all have degree $b$.

Another type of proper $t$-edge colorings, a \emph{cyclic interval
$t$-coloring}, was introduced by de Werra and Solot
\cite{deWerraSolot}. A proper $t$-edge coloring
$\alpha:E(G)\longrightarrow  \{1,\dots,t\}$ of a graph $G$  is
called a \emph{cyclic interval $t$-coloring} if the colors of edges
incident to every vertex $v$ of $G$ either form an interval of
integers or the set $\{1,\ldots,t\}\setminus \{\alpha(e): \text{$e$
is incident to $v$}\}$ is an interval of integers. This notion was
motivated by scheduling problems arising in flexible manufacturing
systems, in particular the so-called \emph{cylindrical open shop
scheduling problem}. Clearly, any interval $t$-coloring of a graph
$G$ is also a cyclic interval $t$-coloring. Therefore all above
mentioned classes of graphs which admit interval edge colorings,
also admit cyclic interval colorings. 
Note that the condition
$\chi'(G)=\Delta(G)$ is necessary for a graph $G$ to admit an
interval edge coloring \cite{AsrKam,AsrKamJCTB}. In contrast with
this, every regular graph $G$ with $\chi'(G)=\Delta(G)+1$ (for
example, $G=K_{2n+1}$) has a cyclic interval
$(\Delta(G)+1)$-coloring. Moreover, for every integer $p\geq 1$
there is a graph $G_p$ with $\chi'(G)= \Delta(G)+p$ which admits a
cyclic interval coloring.
 An example of such a graph is the so-called 
``Shannon's triangle'' 
which is obtained by replacing every edge in a 
triangle $K_3$ with $V(K_3)=\{v_1,v_2,v_3\}$ by $p$ parallel edges.
 Clearly, the maximum degree of this graph is $2p$, 
the chromatic index is $3p$, and a cyclic interval 
coloring of it can be obtained by coloring the edges
 between $v_k$ and $v_{k+1}$ with colors 
$(k-1)p+1,(k-1)p+2,\ldots,(k-1)p+p$, 
for $k=1,2,3$ (where we consider $v_4=v_1$).

Note further that if a graph
has a cyclic interval $t$-coloring, then it does not necessarily have a cyclic interval $(t+1)$-coloring; the complete graph $K_3$ has a cyclic
interval $3$-coloring, but does not admit such a coloring with $4$ colors
(see also Example \ref{example} in Section 4).
Furthermore, the disjoint 
union of graphs with cyclic interval colorings may 
not admit a cyclic interval coloring; 
for instance, the disjoint union of 
$K_3$ and $K_{n}$ does not have a 
cyclic interval coloring for any $n \geq 5$.

Kubale and  Nadolski \cite{KubaleNadol} showed that the problem of
determining whether a given bipartite graph admits a cyclic interval
coloring is $NP$-complete. Some sufficient conditions for a graph to
have a cyclic interval coloring were obtained in
\cite{CarlJToft,Nadol,PetMkhit,deWerraSolot}. Nadolski \cite{Nadol}
proved that any connected graph $G$ with $\Delta(G)\leq 3$ has a cyclic
interval coloring. de Werra and Solot \cite{deWerraSolot} proved
that any outerplanar bipartite graph $G$ has a cyclic interval
$\Delta(G)$-coloring. Petrosyan and Mkhitaryan \cite{PetMkhit}
showed that all complete tripartite graphs are cyclically interval
colorable and conjectured that the same holds for all complete
multipartite graphs. They also proved that if a triangle-free
simple graph $G$ with at least two vertices has a cyclic interval
$t$-coloring, then $t \leq \vert V(G)\vert +\Delta(G)-2$. Casselgren
and Toft \cite{CarlJToft} proved that all $(4,8)$-biregular graphs
admit cyclic interval colorings, and conjectured that the same holds
for all $(a,b)$-biregular graphs. Some other results on this subject
were obtained in \cite{Hertz,KamTrees,KamCycles}. For example,
Altinakar et al. \cite{Hertz} showed that for any graph $G$ with
$\Delta(G)\geq 12$ there is a graph $H_G$ with
$\Delta(H_G)=\Delta(G)$ such that $G$ has an interval $t$-coloring
if and only if $H_G$ has a cyclic interval $t$-coloring. In
\cite{KamTrees,KamCycles}, Kamalian determined all possible values
of $t$ for which simple cycles and trees have a cyclic interval
$t$-coloring.
\bigskip

In the present paper we find new classes of graphs admitting cyclic
interval colorings. We prove that a bipartite graph $G$ with even
maximum degree $\Delta(G)\geq 4$ admits a cyclic interval
$\Delta(G)$-coloring if for every vertex $v$ the degree $d_G(v)$
satisfies either $d_G(v)\geq \Delta(G)-2$ or $d_G(v)\leq 2$. We also
prove that every Eulerian bipartite graph $G$ with maximum degree
$\Delta(G)\leq 8$ has a cyclic interval $\Delta(G)$-coloring. 
Furthermore, some
results are obtained for $(a,b)$-biregular and outerplanar
graphs. 
Finally, we prove that all complete multipartite
graphs admit cyclic interval colorings and consider some problems on
bipartite graphs without cyclic interval colorings.

\section{Cyclic interval colorings of bipartite graphs}\

Before we formulate and prove our results, we introduce some
terminology and notation. A graph $G$ is \emph{cyclically interval
colorable} if it has a cyclic interval $t$-coloring for some
positive integer $t$. The set of all interval cyclically colorable
graphs is denoted by $\mathfrak{N}_{c}$. For a graph $G\in
\mathfrak{N}_{c}$, the least value of $t$ for which it has a cyclic
interval $t$-coloring is denoted by $w_{c}(G)$.
If $\alpha $ is a proper edge coloring of $G$ and $v\in V(G)$, then
$S_{G}\left(v,\alpha \right)$ (or $S\left(v,\alpha \right)$) denotes
the set of colors appearing on edges incident to $v$. 

For two
positive integers $a$ and $b$,
we denote by $\gcd(a,b)$ the greatest common divisor of $a$ and $b$;
if $a\leq b$, then
$\left[a,b\right]$ denotes the interval $\left\{a,\ldots,b\right\}$ of integers. 

A {\em full subdivision} of a graph is a graph obtained by replacing each edge with a path of length $2$.
%
%
A graph $G$ is \emph{Eulerian} if the degree of every vertex of $G$ is
even. Note that if $G$ is connected and Eulerian, then it 
has a closed trail containing every edge of it.
A \emph{$2$-factor} of a graph $G$, where loops are allowed, is a
$2$-regular spanning subgraph of $G$.
We need the following classical result from factor theory
\cite{AkiyamaKano}.\bigskip

\noindent {\bf Petersen's theorem}. {\em Let $G$ be a $2r$-regular
graph (where loops are allowed). Then $G$ can be represented as a
union of edge-disjoint $2$-factors.} \bigskip

The main result of this section is the following:

\begin{theorem}
\label{mytheorem1.2.1} If $G$ is a bipartite graph with
$\Delta(G)=2r$ ($r\geq 2$) and for every $v\in V(G)$, $d_{G}(v)\in
\{1,2,2r-2,2r-1,2r\}$, then $G$ has a cyclic interval $2r$-coloring.
\end{theorem}
\begin{proof} Define an auxiliary graph $G^{\star}$ as follows:
first we take two isomorphic copies $G_{1}$ and $G_{2}$ of the graph
$G$ and join by an edge every vertex with an odd vertex degree in
$G_{1}$ with its copy in $G_{2}$, then for each vertex $u \in
V(G_{1}) \cup V(G_{2})$ of degree $2$, we add $r-1$ loops at $u$,
and for each vertex $v\in V(G_{1}) \cup V(G_{2})$ of degree $2r-2$,
we add a loop at $v$. Clearly, $G^{\star}$ is a $2r$-regular
graph. 
By Petersen's theorem above, $G^{\star}$ can be represented as a union
of edge-disjoint $2$-factors $F_{1},\ldots,F_{r}$. By removing all
loops from $2$-factors $F_{1},\ldots,F_{r}$ of $G^{\star}$, we
obtain that the resulting graph $G^{\prime}$ is a union of edge-disjoint
Eulerian subgraphs $F^{\prime}_{1},\ldots,F^{\prime}_{r}$. Since
$G^{\prime}$ is bipartite, for each $i$ ($1\leq i\leq r$),
$F^{\prime}_{i}$ is a collection of even cycles in $G^{\prime}$, and
we can color the edges of $F^{\prime}_{i}$ alternately with colors
$2i-1$ and $2i$. Let $\alpha$ be the resulting coloring of
$G^{\prime}$. Clearly, $\alpha$ is a proper edge coloring of
$G^{\prime}$ with colors $1,\ldots,2r$, and for each vertex $v\in
V(G^{\prime})$ with $d_{G^{\prime}}(v)=2r$,
$S_{G^{\prime}}(v,\alpha)=[1,2r]$. Since for each vertex $u\in
V(G^{\prime})$ with $d_{G^{\prime}}(u)=2$, there exists  exactly one
Eulerian subgraph $F^{\prime}_{i_{u}}$ such that
$d_{F^{\prime}_{i_{u}}}(u)=2$, we obtain that
$S_{G^{\prime}}(u,\alpha)=[2i_{u}-1,2i_{u}]$ for some $i_{u}$.
Similarly, since for each vertex $v\in V(G^{\prime})$ with
$d_{G^{\prime}}(v)=2r-2$, there exists  exactly one Eulerian subgraph
$F^{\prime}_{i_{v}}$ such that $d_{F^{\prime}_{i_{v}}}(v)=0$, we
obtain that
$S_{G^{\prime}}(v,\alpha)=[1,2r]\setminus[2i_{v}-1,2i_{v}]$ for some
$i_{v}$. Now we can consider the restriction of this proper edge
coloring to the edges of the graph $G$. Clearly, this coloring is a
cyclic interval $2r$-coloring of $G$. $\square$
\end{proof}

From Theorem \ref{mytheorem1.2.1} we deduce a number of corollaries.

\begin{corollary}
\label{corollary1.2.1} Let $H$ be a graph with $\Delta(H)=2r$
($r\geq 2$) where for every  $v\in V(H)$ either $d_H(v)\geq 2r-2$ or
$d_H(v)\leq 2$ holds. Then a full subdivision of $H$ admits a  cyclic
interval $2r$-coloring.
\end{corollary}

A direct consequence of Corollary \ref{corollary1.2.1} is the following:

\begin{corollary}
\label{corollary1.2.2} A full subdivision $G$ of a graph $H$ admits a
cyclic interval $\Delta(G)$-coloring if the maximum degree
$\Delta(H)$ is even and differs from the minimum degree of $H$ by at
most 2.
\end{corollary}

In \cite{PetKhach}, Petrosyan and Khachatrian showed that if a bipartite graph is interval colorable, then a full subdivision of this graph is also interval colorable and conjectured that the same holds for all interval colorable graphs. Recently, Pyatkin \cite{Pyatkin2} confirmed this conjecture. 

\begin{corollary}
\label{corollary1.2.3} If $G$ is a bipartite graph with
$\Delta(G)=4$, then $G\in \mathfrak{N}_{c}$ and $w_{c}(G)=4$.
\end{corollary}
\begin{proof} Let $G$ be a bipartite graph with maximum degree 4.
Clearly, $d_G(v)\in \{1,2,3,4\}$ for every vertex $v\in V(G)$, which, by Theorem \ref{mytheorem1.2.1},
means that $G$ admits a  cyclic interval $4$-coloring. $\square$
\end{proof}

Our next result concerns bipartite graphs with an odd maximum
degree.

\begin{theorem}
\label{mytheorem1.2.2} If $G$ is a bipartite graph with
$\Delta(G)=2r-1$ ($r\geq 2$) and for every $v\in V(G)$, $d_{G}(v)\in
\{1,2,2r-2,2r-1\}$, then $G\in \mathfrak{N}_{c}$ and
$w_{c}(G)\leq2r$.
\end{theorem}
\begin{proof} Let us
construct an auxiliary graph $G^{\star}$ as follows: we take two
isomorphic copies of the graph $G$ and join by an edge one vertex of
degree $2r-1$ with its copy. It is easy to see that $G^{\star}$ is a
bipartite graph with $\Delta(G^{\star})=2r$ ($r\geq 2$) and for
every $v\in V(G^{\star})$, $d_{G^{\star}}(v)\in
\{1,2,2r-2,2r-1,2r\}$. By Theorem \ref{mytheorem1.2.1}, $G^{\star}$
has a cyclic interval $2r$-coloring. Now we can consider the
restriction of this cyclic interval coloring to the edges of the
$G$. This coloring is a cyclic interval coloring of $G$ with
no more than $2r$ colors. Hence, $G\in \mathfrak{N}_{c}$ and
$w_{c}(G)\leq 2r$. $\square$
\end{proof}

Note that Theorems \ref{mytheorem1.2.1} and \ref{mytheorem1.2.2}
imply that every bipartite graph where all 
vertex degrees are in the set $\{1,2,4,5,6\}$ has a cyclic interval edge coloring.\\

Before we move on, we need the following result on bipartite graphs.

\begin{lemma}
\label{ourlemma} If $G$ is a bipartite graph with $\Delta(G)=4$ and
with no vertices of degree $3$, then $G$ has an interval
$4$-coloring $\alpha$ such that for each $v\in V(G)$ with
$d_{G}(v)=2$, either $S_{G}(v,\alpha)=[1,2]$ or
$S_{G}(v,\alpha)=[3,4]$.
\end{lemma}
\begin{proof}
If $G$ has pendant vertices, then we can construct an auxiliary
graph $G^{\prime}$ as follows: we take two isomorphic copies of
$G$ and join by an edge every pendant vertex with its copy. It
is easy to see that $G^{\prime}$ is a bipartite graph with
$\Delta(G^{\prime})=4$ and with no vertices of degree $1$ or $3$.
So, without loss of generality, we may assume that the degree of
every vertex of $G$ is either $4$ or $2$.

Next, we construct an auxiliary graph $G^{\star}$ with loops as
follows: for each vertex $v\in V(G)$ with $d_{G}(v)=2$, we add a
loop at $v$. Clearly, $G^{\star}$ is a $4$-regular graph with loops.
By Petersen's theorem above, $G^{\star}$ can be decomposed into two
edge-disjoint $2$-factors $F_{1}$ and $F_{2}$. By removing all loops
from $2$-factors $F_{1}$ and $F_{2}$ of $G^{\star}$, we obtain that
$G$ can be decomposed into two edge-disjoint Eulerian subgraphs
$F^{\prime}_{1}$ and $F^{\prime}_{2}$. Since $G$ is bipartite, for
each $i$ ($1\leq i\leq 2$), $F^{\prime}_{i}$ is a collection of even
cycles in $G$, and we can color the edges of $F^{\prime}_{i}$
alternately with colors $2i-1$ and $2i$. Let $\alpha$ be the
resulting coloring of $G$. Clearly, $\alpha$ is a proper edge
coloring of $G$ with colors $1,2,3,4$, and for each vertex $u\in
V(G)$ with $d_{G}(u)=4$, $S_{G}(u,\alpha)=[1,4]$. Since for each
$v\in V(G)$ with $d_{G}(v)=2$, either $d_{F^{\prime}_{1}}(v)=0$ or
$d_{F^{\prime}_{2}}(v)=0$, we obtain that either
$S_{G}(v,\alpha)=[3,4]$ or $S_{G}(v,\alpha)=[1,2]$. $\square$
\end{proof}

It follows from Theorem \ref{mytheorem1.2.1} that every Eulerian bipartite
graph of maximum degree at most $6$ has a cyclic interval coloring. Next,
we prove that a stronger proposition is true.

\begin{theorem}
\label{mytheorem1.2.3} Every Eulerian bipartite graph $G$ with
maximum degree at most $8$ has a cyclic interval
$\Delta(G)$-coloring.
\end{theorem}
\begin{proof}
    Let $G$ be an  Eulerian bipartite graph with 
		maximum degree at most $8$.
    If $G$ has maximum degree at most $6$, then the result follows
    from  Theorem  \ref{mytheorem1.2.1}, so we may assume that
    $G$ has maximum degree $8$.
    From $G$ we form a new graph $H$ by splitting each vertex
    $v$ of degree $6$ into two new vertices $v'$ and $v''$,
    where $v'$ has degree $2$ and $v''$ has degree $4$. The partitioning of 
		edges in this splitting  is arbitrary,
    other then ensuring that each vertex receives the correct degree.
    Observe that the resulting graph 
		$H$ is bipartite and the degree of
    each vertex in $H$ is either $2$, $4$ or $8$.

    Note that some components of $H$ might only contain vertices of degree $2$.
    Let $H'$ be the subgraph of $H$ containing every component of $H$
    where all vertices have degree $2$, and set $\hat H = H- V(H')$.

    From $\hat H$ we form a new graph $K$ 
		by replacing every maximal path,
    where all the internal vertices have degree $2$, by an edge
    joining the endpoints of the path; we call such a path in
    $\hat H$ {\em reducible}.
    In the resulting graph $K$
    every vertex has degree $4$ or $8$. Moreover, $K$ may contain loops
    (and multiple edges).

    Since every vertex degree in $K$ is divisible by $4$,
    $K$ has an even number of edges,
		and so $K$ has an Eulerian trail $T$
    with an even number of edges. (If $K$ contains loops, then we choose $T$
    in such a way that all loops at a particular vertex $v$ are traversed
    the first time that we visit $v$.)

    We color the edges of $T$ alternately with colors
    ``Blue'' and ``Red'' in such a way that every vertex is incident with
    equally many Red and Blue edges (where possible loops are counted twice).

    Since every reducible path in $\hat H$ corresponds to a single edge in $K$,
    the edge coloring of $K$ defines an edge coloring 
		of $\hat H$ in the following
    way:

    \begin{itemize}

        \item for every edge in $\hat H$ that is in $K$, we retain the color of
        this edge;

        \item for each reducible path $P$ in $\hat H$, color every edge in $P$
        with the color of the corresponding edge of $K$.

    \end{itemize}

    Next, we extend this edge coloring to $H'$ by coloring every edge in
    this graph by the color Red. Denote the obtained edge coloring
    of $H$ by $\varphi$.

    It is straightforward to see that the coloring $\varphi$ of $H$ satisfies
    the following:

    \begin{itemize}

        \item every vertex of degree $2$ in $H$ is incident with two
        edges of the same color;

        \item every vertex of degree $4$ in $H$ is incident with
        two Red and two Blue edges;

        \item every vertex of degree $8$ in $H$ is incident with
        four Red edges and four Blue edges.

    \end{itemize}
    Furthermore, since there is a one-to-one correspondence between edges of $G$
    and $H$, the coloring $\varphi$ induces an edge coloring $\varphi'$ of $G$
    such that

        \begin{itemize}

        \item every vertex of degree $2$ in $G$ is incident with two
        edges of the same color;

        \item every vertex of degree $4$ in $G$ is incident with
        two Red and two Blue edges;

        \item every vertex of degree $6$ in $H$ is incident with four
        Red edges and two Blue edges, or two Red edges and four Blue edges;

        \item every vertex of degree $8$ in $H$ is incident with
        four Red edges and four Blue edges.

    \end{itemize}

    The Blue edges in $G$ induces a subgraph $G_1$ of $G$, and the Red
    edges in $G$ induces a subgraph $G_2$ of $G$; so $G$ is the
    edge-disjoint union of the graphs $G_1$ and $G_2$.
    Moreover, for each $G_i$ ($i=1,2$):

    \begin{itemize}

        \item every vertex of degree $2$ in $G$ has either degree 
				$0$ or $2$ in $G_i$;

        \item every vertex of degree $4$ in $G$ has degree $2$ in $G_i$;

        \item every vertex of degree $6$ in $G$ has degree $4$ or $2$ in $G_i$

        \item every vertex of degree $8$ in $G$ has degree $4$ in $G_i$.

    \end{itemize}

    Hence each of the subgraphs $G_1$ and $G_2$ is a bipartite graph
    where every vertex has degree $4$ or $2$. By Lemma \ref{ourlemma},
    each $G_i$ has an interval $4$-coloring $f_i$ such that for every
    vertex $v$ of degree $2$ in $G_i$,
    $S_{G_{i}}(v,f_i)=[1,2]$ or $S_{G_{i}}(v,f_i)=[3,4]$.  From $f_1$
    we define a new edge coloring $g_1$ of $G_1$ by replacing colors
    $3$ and $4$ by colors $5$ and $6$, respectively; from $f_2$ we
    define a new edge coloring $g_2$ of $G_2$ by replacing colors
    $1$ and $2$ by colors $7$ and $8$, respectively.
    It is straightforward to verify that the colorings $g_1$ and $g_2$ together
    constitute a cyclic interval $8$-coloring of $G$.
 $\square$
\end{proof}
\medskip

We note that the above result is almost sharp, since there is an
Eulerian bipartite graph with six vertices and maximum degree $12$
without a cyclic interval coloring (see Fig. \ref{The graph H_{3,4}}
in section 5). It is thus an interesting open problem if the
condition of maximum degree at most $8$ can be replaced by $10$ in
the above theorem.

\begin{corollary}
\label{corollary1.2.5} A bipartite graph $G$ where all vertex
degrees are in the set $\{1,2,4,6,7,8\}$ has a cyclic interval
coloring.
\end{corollary}
\begin{proof} If $G$ is an Eulerian graph, the existence of a cyclic
interval coloring is evident by Theorem \ref{mytheorem1.2.3}.
Suppose that $G$ has some vertices with odd degrees. Define an
auxiliary graph $G^{\star}$ as follows: we take two isomorphic
copies $G_{1}$ and $G_{2}$ of the graph $G$ and join by an edge
every vertex with an odd vertex degree in $G_{1}$ with its copy in
$G_{2}$. Clearly, $G^{\star}$ is a bipartite Eulerian graph with
maximum degree at most 8. Therefore, by Theorem
\ref{mytheorem1.2.3}, $G^{\star}$ has a cyclic interval coloring. It
is not difficult to see that the restriction of this coloring to the
edges of $G$ is a cyclic interval coloring of $G$. $\square$
\end{proof}

Let us now consider $(a,b)$-biregular graphs.
The following result is an evident corollary of Theorem
\ref{mytheorem1.2.1}.

\begin{corollary}
\label{mycorollary1.2.6} If $G$ is a $(2r-2,2r)$-biregular ($r\geq
2$) graph, then $G\in \mathfrak{N}_{c}$ and $w_{c}(G)=2r$.
\end{corollary}

Our next result establishes a connection between the existence of
cyclic interval colorings for $(a,b)$-biregular and
$(a,b-1)$-biregular graphs.

\begin{theorem}
\label{mytheorem1.2.4} If every $(a,b)$-biregular ($a<b$) graph has
a cyclic interval $b$-coloring and $\gcd(a,b-1)=1$, then every
$(a,b-1)$-biregular graph has a cyclic interval $b$-coloring.
\end{theorem}
\begin{proof} Let $G$ be an $(a,b-1)$-biregular ($a<b$) bipartite
graph with bipartition $(X,Y)$. Clearly, $a\vert X\vert=(b-1)\vert
Y\vert$. Since $\gcd(a,b-1)=1$, we have $\vert Y\vert=a k$ for
some integer $k$. Let $Y=\{y_{1},\ldots,y_{a k}\}$. Now we
define an auxiliary graph $G^{\prime}$ as follows:
\begin{center}
$V\left(G^{\prime}\right)=X\cup X' \cup Y$ and
$E\left(G^{\prime}\right)=E(G)\cup E^{\prime}$,\\
where
$X^{\prime}=\left\{x^{\prime}_{1},\ldots,x^{\prime}_{k}\right\}$ and
$E^{\prime}=\left\{x^{\prime}_{i}y_{a(i-1)+1},\ldots,
x^{\prime}_{i}y_{a i}\colon\, 1\leq i\leq k\right\}$.
\end{center}
Clearly, $G^{\prime}$ is an $(a,b)$-biregular bipartite graph with
bipartition $(X \cup X',Y)$, and since $G^{\prime}$ is
$(a,b)$-biregular, $G^{\prime}$ has a cyclic interval $b$-coloring.
It is not difficult to see that the restriction of this edge
coloring to the edges of $G$ induces a cyclic interval $b$-coloring.
$\square$
\end{proof}

Since all $(3,6)$-biregular and $(4,8)$-biregular graphs have cyclic
interval $6$- and $8$-colorings, respectively \cite{CarlJToft}, we
deduce the following two consequences from Theorem
\ref{mytheorem1.2.4}. The first one was first obtained in
\cite{CasPetToft} (using essentially the same proof).

\begin{corollary} \cite{CasPetToft}
\label{mycorollary1.2.7} If $G$ is a $(3,5)$-biregular graph, then
$G\in \mathfrak{N}_{c}$ and $w_{c}(G)\leq 6$.
\end{corollary}

\begin{corollary}
\label{mycorollary1.2.8} If $G$ is a $(4,7)$-biregular graph, then
$G\in \mathfrak{N}_{c}$ and $w_{c}(G)\leq 8$.
\end{corollary}
\bigskip


\section{Cyclic interval colorings of outerplanar graphs}\

Let us now consider outerplanar graphs. We conjecture that
all connected outerplanar graphs have cyclic interval colorings, 
and we prove
this conjecture for simple graphs with maximum degree at most 4.
For the proof, we shall use
the fact that every simple $2$-connnected
outerplanar graph with maximum degree $3$ has an interval coloring with
$3$ or $4$ colors \cite{PetrosyanPlanar}.

\begin{theorem}
\label{mytheorem2.1}
    If $G$ is a simple connected outerplanar graph with
    maximum degree $\Delta(G)\leq 4$,
    then $G$ has a cyclic interval coloring.
\end{theorem}

\begin{proof}
    If $\Delta(G) \leq 3$, then $G$ has a cyclic interval coloring by the
    result of Nadolski \cite{Nadol} so it suffices to prove the
    theorem when $\Delta(G) =4$.
		
		We shall prove the theorem using the following two claims.
		
		\begin{claim}
		\label{cl:1}
			Every simple $2$-connected outerplanar  graph of maximum degree
			$4$ has a cyclic interval $5$-coloring.		
		\end{claim}
		
		\begin{proof}
    Let $G$ be a simple $2$-connected  outerplanar  graph. Then $G$ has a
    Hamiltonian cycle $C$, implying that
		$G-E(C)$ is a simple graph with maximum degree $2$.
		We define a proper edge coloring $\alpha$ of 
		$G-E(C)$ as follows: for each component
		which is a path or an even cycle we color the edges of it by colors
		$1$ and $3$ alternately; for each component which is an odd cycle
		we color one of the edges of it by color $2$, and the rest of the
		edges in the component by colors $1$ and $3$ alternately.
		%

    Suppose first that $|V(G)|$ is even; then we may properly
    color the edges of $C$ using colors $4$ and $5$. This edge coloring
    along with the coloring $\alpha$ of $G-E(C)$ constitute a
    cyclic interval coloring of $G$.

    Suppose now that $|V(G)|$ is odd. We consider some different cases.

    \bigskip
    \noindent
    {\bf Case 1}. $C$ contains two consecutive vertices both of which
    have degree $4$ in $G$:

    Let $u$ and $v$ be two vertices of degree $4$ in $G$, which are
    consecutive on $C$.
    Suppose first that $u$ and $v$ lie in different components
    of $G-E(C)$. Then we may without loss of generality assume
    that both $u$ and $v$ are incident with two edges colored $1$
    and $3$ under $\alpha$, respectively (possibly by shifting colors
    along cycles or paths in $G-E(C)$). We define an edge coloring $\beta$
    of $C$ by coloring $uv$ with color $2$ and coloring the rest of
    the edges of $C$ by colors $4$ and $5$ alternately. It is
    straightforward that $\alpha$ and $\beta$ together constitute a cyclic
    interval coloring of $G$.

    Suppose now that $u$ and $v$ belong to the same component $S$ of
    $G-E(C)$. Since $G$ is an outerplanar graph, this implies that
    $S$ is a path and thus its edges are colored
    by colors $1$ and $3$ under $\alpha$; we obtain a cyclic interval
    coloring by coloring the edges of $C$ as in the preceding case.

    \bigskip
    \noindent
    {\bf Case 2}. $C$ does not contain two consecutive vertices
    both of which have degree $4$ in $G$:

    Suppose first that there are vertices $u$ and $v$ of degree
    $3$ and $4$ in $G$, respectively, which are consecutive on $C$.
    If $u$ and $v$ lie in different components
    of $G- E(C)$, then we may assume that $v$ is incident with two edges
    colored $1$ and $3$, respectively, under $\alpha$. Moreover,
        $u$ is incident
    with an edge colored $1$ or $3$ under $\alpha$. We may thus proceed
    by coloring $uv$ with color $2$ and then coloring the rest
		of $C$ by colors $4$ and $5$ alternately, starting
		with color $5$ at $u$ if $u$ is incident with an edge colored $1$,
		and starting with color $4$ at $u$ otherwise.

    If $u$ and $v$ lie in the same component $S$ of $G-E(C)$, then
    $S$ is a path, and thus
    the edges of $S$ are colored alternately with colors $1$ and $3$ under
    $\alpha$. Thus, we may proceed as in the preceding paragraph.

    Suppose now that there are no vertices $u$ and $v$ of degree
    $3$ and $4$ in $G$, respectively, 
		which are consecutive on $C$.
    Then, if $u$ has degree $4$, then any neighbor of $u$ on $C$
    has degree $2$.
    Suppose first that $G-E(C)$ contains at least one cycle
		or a path of length at least $3$.
		Let $x$ be a vertex of degree $2$ on this cycle or path,
		let $y$ be a neighbor of $x$ on $C$,
		and let $z$ be a neighbor of $x$ in $G-E(C)$
		which has degree $2$ in $G-E(C)$.
    We construct a proper edge coloring 
		$\alpha'$ of $G-E(C)$ from $\alpha$ by
    recoloring the edges
    of the component containing $x$ by colors $1,2,3$
    in such a way that $xz$ is colored $2$, the
    other edge incident with $x$ is colored $1$, and all other edges
		of this component is colored by $1$ and $3$ alternately.
		By coloring the edge $xy$ with color $3$,
    and the rest of the edges of $C$ by colors $4$ and $5$, alternately,
        and starting with color $4$ at $y$,
    we clearly obtain a cyclic interval coloring of $G$.

    Suppose now that $G-E(C)$ is a vertex-disjoint union of paths
		all of which have length at most $2$.
		Since $G$ is outerplanar and has maximum
		degree $4$ in $G$, there is a path $P=xuv$ in 
		$G-E(C)$ of length $2$. Denote by $Q$ the path in $C$
		with origin $u$ and terminus $v$ that contains $x$ as an inner vertex.
		We color $Q$ by colors $2$ and $3$ alternately, and starting with
		color $2$ at $u$; we color all other edges of $C$ with colors
		$5$ and $4$ alternately and starting with color $5$ at $u$.
		%
		A component $S$ of $G-E(C)$ we color by $1$ and $3$ 
		alternately if
		both endpoints of $S$ are in $V(C) \setminus V(Q)$;
		if the endpoints of $S$ are in $Q$ we color it by $1$ and $4$
		alternately,
		except for the path $xuv$, where
		we color $xu$ with $1$ and $uv$
		with $3$ ($4$) if $Q$ has odd (even) length. 
		The resulting coloring is a cyclic interval $5$-coloring.
		$\square$
 \end{proof}

	We shall also need the following claim:
	
	\begin{claim}
	\label{cl:2}
		Every simple $2$-connected outerplanar graph $G$
		of maximum degree at most $3$
		has a 
		cyclic 
		interval $3$-coloring $\alpha$   
		with the property that
		there is at most one vertex $v \in V(G)$ 
		such that $S_G(v,\alpha)$ is not
		an interval, or an interval coloring with at most $4$ colors. 
		Moreover, if $G$ is not interval colorable
		using at most $4$ colors, then
		for any vertex $v \in V(G)$ of degree $2$, we can choose
		the coloring $\alpha$ so that $v$ is the unique vertex
		with the property that $S_G(v,\alpha)$ is not interval.
	\end{claim}
	\begin{proof}
		If $G$ has maximum degree at most $2$, then trivially $G$
		has a cyclic interval coloring with the required property 
		or an interval coloring (and thus
		also a cyclic interval coloring).
		
		If $G$ has maximum degree $3$, then the claim follows from the
		result of Petrosyan \cite{PetrosyanPlanar} that every simple
		$2$-connected outerplanar graph of maximum degree $3$ 
		has an interval coloring with at most $4$ colors.
		%
		%
	$\square$
 \end{proof}

		We now finish the proof of the theorem by proving 
		that every simple outerplanar graph with maximum degree $4$
		has a cyclic interval coloring. The proof is by induction on
		the number of blocks of $G$. Since any tree is interval colorable, 
		we may assue that there is some cycle of $G$.
	
		If $G$ has only one block, then the result follows from Claims
		\ref{cl:1} and \ref{cl:2}. 
		Suppose now that $G$ has two blocks $F_1$ and $F_2$, and that
		$v$ is the common vertex of these blocks.
		We shall prove that $G$ has a cyclic interval $5$-coloring.
		
		Suppose first that one of the blocks, say $F_2$, 
		consist of a single edge. Since any interval coloring with at most
		$5$ colors is also a cyclic interval $5$-coloring, it follows from
		Claims \ref{cl:1} and \ref{cl:2} that
		$F_1$ either has a 
		\begin{itemize}
		 
		\item cyclic interval $3$-coloring $\alpha$,
		with the property that $v$
		is the only vertex $u$ such that $S(u,\alpha)$ 
		may not be an interval,
			
		\item a cyclic interval $5$-coloring.
		
		\end{itemize}
		It is straightforward to verify that in both cases
		we may color the edge of $F_2$ to obtain a cyclic
		interval $5$-coloring of $G$. 
		
		Suppose now that both $F_1$ and $F_2$ have maximum degree
		at least $2$. 
		Since $G$ has maximum degree $4$, this 
		implies that two edges of
		$F_1$, and two edges of $F_2$, are incident with $v$.
		By Claims \ref{cl:1} and \ref{cl:2}, for $i \in \{1,2\}$,
		$F_i$ either has a 
		\begin{itemize}
		 
		\item cyclic interval $3$-coloring $\alpha_i$,
		with the property that $v$
		is the only vertex $u$ such that $S(u,\alpha_i)$ may not be an interval,
		or 
		
		\item a cyclic interval $5$-coloring $\alpha_i$.
		
		\end{itemize}

		If both $F_1$ and $F_2$ have cyclic interval
		$3$-colorings, $\alpha_1$ and $\alpha_2$, respectively,
		then we may assume that $S(v,\alpha_1) = S(v,\alpha_2) = \{1,3\}$,
		and we define a new
		coloring $\alpha'_2$ of $F_2$ by setting $\alpha_2'(e) = \alpha_2(e)+1$
		for any edge $e$ of $F_2$.
		Taking $\alpha_1$ and $\alpha_2'$ together we obtain an interval coloring
		of the graph, and hence a cyclic interval $5$-coloring.
		
		
		Suppose now that one of $F_1$ and $F_2$ has a
		cyclic interval $5$-coloring. Assume e.g. that
		$\alpha_1$ is a cyclic interval
		$5$-coloring of $F_1$. Then
		we may rotate the colors of $\alpha_1$ modulo $5$ 
		to obtain a coloring
		$\alpha'_1$ so that
		$S(v,\alpha'_1) \cap S(v,\alpha_2) = \emptyset$, and thus 
		$\alpha'_1$ and $\alpha_2$ taken together form a cyclic interval
		$5$-coloring.
		
		\bigskip

		Now assume that $G$ has several blocks $F_1,F_2,\dots, F_n$
		and that we have constructed a cyclic interval $5$-coloring $\alpha$
		of the connected subgraph $H_r$ of $G$
		consisting of blocks $F_1, F_2,\dots, F_r$.
		Suppose that $F_{r+1}$ has exactly one vertex in common with $H_r$.
		We complete the induction step by proving that 
		there is a cyclic interval $5$-coloring of the union $H_{r+1}$
		of $H_r$ and $F_{r+1}$.
		
		If $F_{r+1}$ consist of a single edge, then the result is
		trivial. Suppose now that $F_{r+1}$ has maximum degree at
		least $2$, and let $v$ be the common vertex of $H_r$ and $F_{r+1}$.
		It follows that $2$ edges of $H_r$, and two edges of $F_{r+1}$,
		are incident with $v$. Moreover, 
		by rotating the colors in $\alpha$ modulo $5$,
		we may assume that $S_{H_r}(v,\alpha) =\{1,2\}$.
		
		By Claims \ref{cl:1} and \ref{cl:2}, $F_{r+1}$
		has a cyclic interval $3$-coloring $\beta$
		satisfying that $v$ is the only vertex $u$ with the property
		that $S(u,\beta)$ may not be an interval,
		or a cyclic interval $5$-coloring. If the former holds, then by 
		defining the coloring $\beta'$ by setting
		$\beta'(e) = \beta(e) +2$, we obtain a cyclic interval $5$-coloring
		of $H_{r+1}$ by taking $\beta'$ and $\alpha$ together. 
		%
		If $F_{r+1}$ has a cyclic interval $5$-coloring $\beta$, 
		then by rotating the colors of
		$\beta$ we obtain a coloring $\beta'$ such that 
		$S(v,\beta') \cap S(v, \alpha) =\emptyset$, and thus $\alpha$ and
		$\beta'$ together form a cyclic interval $5$-coloring of $H_{r+1}$.
		%
	%
	$\square$
 \end{proof}

\bigskip

\section{Cyclic interval colorings of complete multipartite graphs}\

A graph $G$ is called a complete $r$-partite ($r\geq 2$) graph if
its vertices can be partitioned into $r$ nonempty independent sets
$V_1,\ldots,V_r$ such that each vertex in $V_i$ is adjacent to all
the other vertices in $V_j$ for $1\leq i<j\leq r$. Let
$K_{n_{1},n_{2},\ldots,n_{r}}$ denote a complete $r$-partite graph
with independent sets $V_1,V_2,\ldots,V_r$ of sizes
$n_{1},n_{2},\ldots,n_{r}$. We set $n=\sum_{i=1}^{r}n_{i}$.

In \cite{PetMkhit}, it was conjectured that all complete
multipartite graphs are cyclically interval colorable. Here we prove
this conjecture.

\begin{theorem}
\label{mytheorem3.1} For any $n_{1},n_{2},\ldots,n_{r}\in
\mathbb{N}$, the graph $K_{n_{1},n_{2},\ldots,n_{r}}$ has a cyclic
interval $n$-coloring.
\end{theorem}
\begin{proof}
We have that  
$n=\vert V\left(K_{n_{1},n_{2},\ldots,n_{r}}\right)\vert=
\sum_{j=1}^{r}n_{j}$.
 For $0\leq i\leq r$, define a sum $\sigma(i)$ as follows:
\begin{center}
$\sigma(i)=\left\{
\begin{tabular}{ll}
$0$, & if $i=0$,\\
$\sum_{j=1}^{i}n_{j}$, & if $1\leq i\leq r$.\\
\end{tabular}%
\right.$
\end{center}
Clearly, $\sigma(r)=n$. 
Let
$V_{i}=\left\{v_{\sigma(i-1)+1},\ldots,v_{\sigma(i)}\right\}$ for
$1\leq i\leq r$.

Define an edge coloring $\alpha$ of $K_{n_{1},n_{2},\ldots,n_{r}}$
as follows: for $v_{i}v_{j}\in E\left(
K_{n_{1},n_{2},\ldots,n_{r}}\right)$, let

\begin{center}
$\alpha\left(v_{i}v_{j}\right)=\left\{
\begin{tabular}{ll}
$(i+j)\pmod {n}$, & if $i+j\neq n$,\\
$n$, & otherwise.\\
\end{tabular}%
\right.$
\end{center}

Let us prove that $\alpha$ is a cyclic interval $n$-coloring of
$K_{n_{1},n_{2},\ldots,n_{r}}$.

First note that in the coloring $\alpha$ every color is used on some
edge. Next let $v_{i}\in V_{l}$, where $1\leq l\leq r$. If $l=1$,
then, by the definition of $\alpha$, we have
$S\left(v_{i},\alpha\right)=[i+\sigma(1)+1,n]\cup
[1,i]=[i+n_{1}+1,n]\cup [1,i]$. If $1<l\leq r$, then, by the
definition of $\alpha$, we have that $S\left(v_{i},\alpha\right)$
contains colors $1,\ldots,n$ except for $i+\sigma(l-1)+1,\ldots
,i+\sigma(l)$, where colors are taken modulo $n$ and with $n$
instead of $0$. This implies that $\alpha$ is a cyclic interval
$n$-coloring of $K_{n_{1},n_{2},\ldots,n_{r}}$. $\square$
\end{proof}

\begin{corollary}
\label{mycorollary3.1} For any $n_{1},n_{2},\ldots,n_{r}\in
\mathbb{N}$, we have $K_{n_{1},n_{2},\ldots,n_{r}}\in
\mathfrak{N}_{c}$ and
\begin{center}
$w_{c}(K_{n_{1},n_{2},\ldots,n_{r}})\leq \sum_{i=1}^{r}n_{i}$.
\end{center}
\end{corollary}

We will show that the upper bound in Corollary \ref{mycorollary3.1}
is sharp for some complete multipartite graphs. But first we need
the following result.

\begin{theorem}
\label{mytheorem3.2} If for a graph $G$, there exists a number $d$
such that $d$ divides $d_{G}(v)$ for every $v\in V(G)$ and $d$ does
not divide $\vert E(G)\vert$, then $G$ has no cyclic interval
$d k$-coloring for every $k\in \mathbb{N}$.
\end{theorem}
\begin{proof}
Suppose, to the contrary, that $G$ has a cyclic interval
$d k$-coloring $\alpha$ for some $k\in \mathbb{N}$. We call an edge
$e\in E(G)$ a \emph{$d$-edge} if $\alpha(e)=d l$ for some $l\in
\mathbb{N}$. Since  
$d$ divides $d_{G}(v)$ for every $v\in V(G)$, and $\alpha$ is a
cyclic interval coloring, we have that for any $v\in V(G)$, the set
$S\left(v,\alpha\right)$ contains exactly $\frac{d_{G}(v)}{d}$
$d$-edges. Now let $m_{d}$ be the number of $d$-edges in $G$. 
Then $m_{d}=\frac{1}{2}\sum\limits_{v\in
V(G)}\frac{d_{G}(v)}{d}=\frac{\vert E(G)\vert}{d}$. Hence, $d$
divides $\vert E(G)\vert$, which is a contradiction. ~$\square$
\end{proof}

\begin{corollary}
\label{mycorollary3.2} If $G$ is an Eulerian graph and $\vert
E(G)\vert$ is odd, then $G$ has no cyclic interval $t$-coloring for
every even $t$.
\end{corollary}

\begin{example}
\label{example}
{\em The graph consisting of 
three edge-disjoint triangles, where any two triangles have
the same common vertex $v$, has a cyclic interval
$7$-coloring: we color the first triangle by colors 
$1,2,3$ so that $1$ and $3$ appear at $v$, color the second triangle by colors $2,3,4$ so that $2$ and $4$ appear at $v$, and we color the third triangle by colors $5,6,7$ so that $5$ and $7$ appear at $v$. 
This yields a cyclic interval $7$-coloring of the graph. 
However, by Corollary \ref{mycorollary3.2}, this graph
does not admit a cyclic interval $8$-coloring. }
\end{example}

The next result shows that the upper bound in Corollary
\ref{mycorollary3.1} is sharp.

\begin{corollary}
\label{mycorollary3.3} If $r,n_{2},\ldots,n_{r}$ and $\vert
E\left(K_{1,n_{2},\ldots,n_{r}}\right)\vert$ are odd, then
\begin{center}
$w_{c}(K_{1,n_{2},\ldots,n_{r}})=1+ \sum_{i=2}^{r}n_{i}$.
\end{center}
\end{corollary}

\begin{proof}
Clearly, the graph $G=K_{1,n_{2},\ldots,n_{r}}$ is Eulerian with the
maximum degree $\Delta(G)=\sum_{i=2}^{r}n_{i}$, which implies that  
$w_c(G)\geq \sum_{i=2}^{r}n_{i}$. 
By Corollary \ref{mycorollary3.2}, $G$ has no cyclic
interval $\Delta(G)$-coloring, since $|E(G)|$ is odd. This and
Corollary \ref{mycorollary3.1} imply that $w_{c}(G)=1+
\sum_{i=2}^{r}n_{i}$.~$\square$
\end{proof}
\bigskip

\section{Bipartite graphs without cyclic interval colorings}\

In this section we present some final observations; we give an
example of a simple bipartite graph with maximum degree $14$ without
a cyclic interval coloring. In terms of maximum degree, this is an
improvement of the smallest previously known  bipartite example which
has maximum degree $28$ \cite{Nadol}. We also give a similar example
for bipartite graphs with multiple edges. 
Our examples use graphs
without interval colorings earlier constructed in \cite{PetKhach}.
For non-bipartite examples of graphs without cyclic interval
colorings, see \cite{PetMkhit}.

So let us consider the graph $G$ in Fig. \ref{cyc-int-non sim
bigraph}. Clearly, $\Delta(G)=14$ and $\vert V(G)\vert=21$.

\begin{figure}[ht]
\begin{center}
\includegraphics[width=20pc,height=17pc]{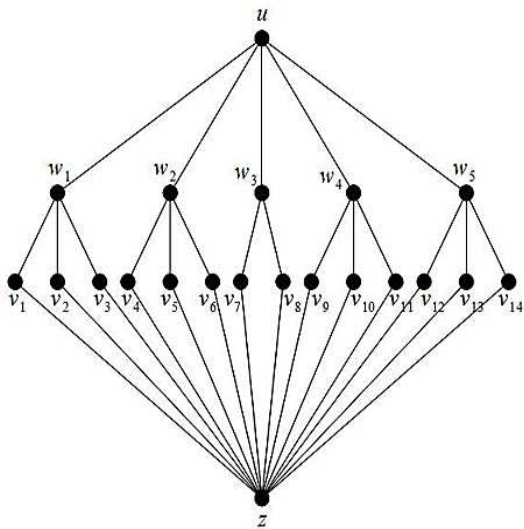}\\
\caption{The cyclically interval non-colorable simple bipartite
graph $G$.}\label{cyc-int-non sim bigraph}
\end{center}
\end{figure}

\begin{proposition}
\label{proposition4.1} The graph $G$
in Figure 1 is not cyclically interval colorable.
\end{proposition}
\begin{proof}
Suppose, to the contrary, that $G$ has a cyclic interval
$t$-coloring $\alpha$ for some $t\geq 14$. Assume further that
$\alpha$ is such a coloring using a minimum number of colors; then
every color appears on at least one edge. 
We may assume that $S_{G}(u,\alpha)=[1,5]$ and
$\alpha(uw_{j})=i_j$ and $1\leq i_j\leq 5$, for $j=1,\ldots,5$.
Let $N_{w_j}$ denote the neighbors of $w_j$ in $\{v_1,\dots, v_{14}\}$
and consider the edges incident with $w_j$ and some vertex from $N_{w_j}$.
Any such edge
can receive colors only from the
set $\{i_j-3, i_j-2, \dots, i_j+3\}$, and therefore 
any edge joining $z$ with a vertex of $N_{w_j}$
can receive colors only from the set $\{i_j-4, i_j-3 \dots, i_j+4\}$,
where $(-)$ and $(+)$ denote subtraction modulo $t$ and addition
modulo $t$, respectively. This implies that none of the edges
incident with a vertex from $N_{w_j}$ is
colored $10$. 
$\square$
\end{proof}

Similarly, it can be shown that for any positive integer $\Delta\geq
14$, there exists a simple bipartite graph $G$ such that $G\notin
\mathfrak{N}_{c}$ and $\Delta(G)=\Delta$. On the other hand, by
Corollary \ref{corollary1.2.3} every bipartite graph with maximum
degree $4$ has a cyclic interval $4$-coloring. So, it is natural to
consider the following:

\begin{problem}\label{ourproblem1}
Is there a simple bipartite graph $G$ such that $5\leq \Delta(G)\leq 13$
and $G\notin \mathfrak{N}_{c}$?
\end{problem}
\bigskip

\begin{figure}[h]
\begin{center}
\includegraphics[width=15pc,height=20pc]{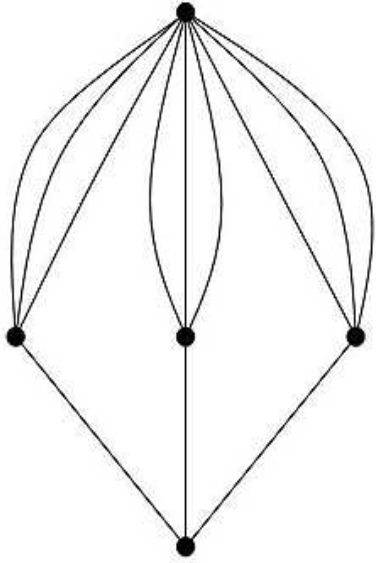}\\
\caption{The cyclically interval non-colorable bipartite graph
$H$.}\label{cyc-int-non bigraph}
\end{center}
\end{figure}

Let us now consider the corresponding problem for bipartite graphs
with multiple edges.
It was proved in \cite{PetKhach} that all
bipartite graphs with at most four vertices have interval colorings;
so all these graphs admit cyclic interval colorings. On the other
hand, it is easy to see that the bipartite graph $H$ with $\vert
V(H)\vert=5$ and $\Delta(H)=9$ shown in Fig. \ref{cyc-int-non
bigraph} has no cyclic interval coloring. We now prove a more
general result.

Let us define graphs $H_{p,q}$ ($p,q\in \mathbb{N}$) as follows:
$V\left(H_{p,q}\right)=\{x,y_{1},\ldots,y_{q},z\}$ and
$E\left(H_{p,q}\right)$ contains pairs of vertices $x$ and $y_{i}$,
which are joined by $p$ edges and the edges $y_{i}z$ for $1\leq
i\leq q$. Clearly, $H_{p,q}$ is a connected bipartite graph with
$\vert V(H_{p,q})\vert=q+2$, $\Delta (H_{p,q})=d(x)=p q$, and
$d(z)=q$, $d(y_{i})=p+1$, $i=1,\ldots,q$. Fig. \ref{The graph
H_{3,4}} shows the graph $H_{3,4}$.

\begin{figure}[h]
\begin{center}
\includegraphics[width=20pc,height=15pc]{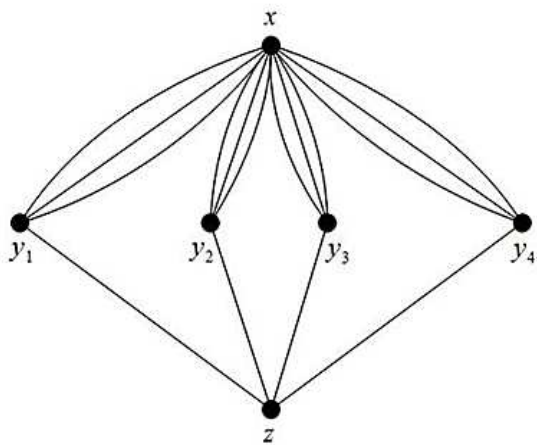}\\
\caption{The graph $H_{3,4}$.}\label{The graph H_{3,4}}
\end{center}
\end{figure}

\begin{proposition}
\label{proposition4.2} If $p q > 2p+q$, 
then $H_{p,q}\notin \mathfrak{N}_{c}$.
\end{proposition}
\begin{proof} Suppose, to the contrary, that $H_{p,q}$ has a cyclic
interval $t$-coloring $\alpha$ for some $t\geq p q$. Assume
further that $\alpha$ is such a coloring using a minimum number of
colors; then every color appears on at least one edge. 
We may assume that $S(z,\alpha)=[1,q]$ and
$\alpha(zy_{j})=i_j$ and $1\leq i_j\leq q$, for $j=1,\ldots,q$.
Consider the edges incident with the vertex $y_j$. Clearly, all
parallel edges $xy_{j}$ can receive colors only from the 
set
$\{i_j-p, i_j-p+1,\dots, i_j+p\}$, 
where $(-)$ and $(+)$ denote subtraction modulo $t$
and addition modulo $t$, respectively. Since $p q > 2p+q$, we have that none of the edges $xy_{j}$ and $zy_{j}$ is colored $p q-p$. 
$\square$
\end{proof}

Using similar arguments as in the proof of Proposition
\ref{proposition4.2}, it can be shown that for any positive integer
$\Delta\geq 9$, there exists a bipartite graph $G$ such that
$G\notin \mathfrak{N}_{c}$ and $\Delta(G)=\Delta$. So, it is natural
to consider the following:

\begin{problem}\label{ourproblem2}
Is there a bipartite graph $G$ such that $5\leq \Delta(G)\leq 8$ and
$G\notin \mathfrak{N}_{c}$?
\end{problem}
\bigskip

\begin{acknowledgement}
The authors would like to thank the referees for their helpful
comments and suggestions.

The third author would like to thank Link\"oping University for the hospitality and nice environment. The work of the third author was made possible by a research grant from the Armenian National Science and Education Fund (ANSEF) based in New York, USA.
\end{acknowledgement}


\bigskip

\begin{thebibliography}{99}

\bibitem{AkiyamaKano} J. Akiyama and M. Kano, Factors and Factorizations of Graphs,
proof techniques in Factor Theory, Springer-Verlag Berlin
Heidelberg, 2011.

\bibitem{Hertz} S. Altinakar, G. Caporossi and A. Hertz, On compact $k$-edge-colorings: A polynomial time reduction from linear to cyclic,
Discrete Optimization 8 (2011), 502--512.

\bibitem{AsrCarl} A.S. Asratian and C.J. Casselgren, On interval edge colorings of
$(\alpha,\beta)$-biregular bipartite graphs, Discrete Math. 307
(2006), 1951--1956.

\bibitem{AsrCarlVandWest} A.S. Asratian, C.J. Casselgren, J. Vandenbussche and D.B. West,
Proper path-factors and interval edge-coloring of $\left(3,4\right)
$-biregular bigraphs, J. Graph Theory 61 (2009), 88--97.

\bibitem{AsrKam} A.S. Asratian and R.R. Kamalian, Interval colorings of edges of a
multigraph, Appl. Math. 5 (1987) 25--34 (in Russian).

\bibitem{AsrKamJCTB} A.S. Asratian and R.R. Kamalian, Investigation on interval
edge-colorings of graphs, J. Combin. Theory Ser. B 62 (1994), 34--43.

\bibitem{CasPetToft} C.J. Casselgren, P.A. Petrosyan, B. Toft, On interval and cyclic
interval edge colorings of $(3,5)$-biregular graphs, 
to appear in Discrete Mathematics.

\bibitem{CarlJToft} C. J. Casselgren, B. Toft, On interval edge colorings of biregular bipartite graphs
with small vertex degrees, J. Graph Theory 80 (2015), 83--97.

\bibitem{Giaro} K. Giaro and M. Kubale, Consecutive edge-colorings of complete and incomplete Cartesian products of graphs,
Congr. Numer. 128 (1997) 143--149.

\bibitem{GiaroKub} K. Giaro and M. Kubale,  Compact scheduling of zero-one time operations in multi-stage system,
Discrete Applied Math. 145 (2004) 95--103.

\bibitem{Hansen} H.M. Hansen, Scheduling with minimum waiting periods, MSc
Thesis, Odense University, Odense, Denmark, 1992 (in Danish).

\bibitem{HanLotToft} D. Hanson, C.O.M. Loten and B. Toft, On interval colorings of
bi-regular bipartite graphs, Ars Combin. 50 (1998), 23--32.

\bibitem{Kampreprint} R.R. Kamalian, Interval colorings of complete bipartite graphs
and trees, preprint, Comp. Cen. of Acad. Sci. of Armenian SSR,
Yerevan, 1989 (in Russian).

\bibitem{KamDiss} R.R. Kamalian, Interval edge colorings of graphs, Doctoral
Thesis, Novosibirsk, 1990.

\bibitem{KamMirum} R.R. Kamalian and A.N. Mirumian, Interval edge-colorings of
bipartite graphs of some class, Dokl. NAN RA 97 (1997), 3--5 (in
Russian).

\bibitem{KamTrees} R.R. Kamalian, On cyclically-interval edge colorings of trees,
Buletinul of Academy of Sciences of the Republic of Moldova,
Matematica 1(68) (2012), 50--58.

\bibitem{KamCycles} R.R. Kamalian, On a number of colors in cyclically interval
edge colorings of simple cycles, Open J. Discrete Math. 3 (2013),
43--48.

\bibitem{KubaleNadol} M. Kubale and A. Nadolski, Chromatic scheduling in a
cyclic open shop, European J. Oper. Res. 164 (2005), 585--591.

\bibitem{Nadol} A. Nadolski, Compact cyclic edge-colorings of graphs, Discrete
Math. 308 (2008), 2407--2417.

\bibitem{PetrosyanPlanar} P.A. Petrosyan, On Interval Edge-Colorings of
Outerplanar Graphs, to appear in Ars Combinatoria

\bibitem{PetDM} P.A. Petrosyan, Interval edge-colorings of complete graphs and
$n$-dimensional cubes, Discrete Math. 310 (2010), 1580--1587.

\bibitem{PetKhachTan} P.A. Petrosyan, H.H. Khachatrian and H.G. Tananyan, Interval
edge-colorings of Cartesian products of graphs I, Discuss. Math.
Graph Theory 33(3) (2013), 613--632.

\bibitem{PetKhach} P.A. Petrosyan and H.H. Khachatrian, Interval non-edge-colorable
bipartite graphs and multigraphs, J. Graph Theory 76 (2014),
200--216.

\bibitem{PetMkhit} P.A. Petrosyan and S.T. Mkhitaryan, Interval cyclic edge-colorings of graphs,
Discrete Mathematics 339 (2016), 1848--1860.

\bibitem{Pyatkin} A.V. Pyatkin, Interval coloring of $\left(3,4\right)
$-biregular bipartite graphs having large cubic subgraphs, J. Graph
Theory 47 (2004), 122--128.

\bibitem{Pyatkin2} A.V. Pyatkin, On an interval $(1,1)$-coloring of
incidentors of interval colorable graphs,
Journal of Applied and Industrial Mathematics 9 (2015), 271--274.

\bibitem{Seva} S.V. Sevast'janov, Interval colorability of the edges of a
bipartite graph, Metody Diskret. Analiza 50 (1990), 61--72 (in
Russian).

\bibitem{deWerraSolot} D. de Werra and Ph. Solot, Compact cylindrical chromatic scheduling,
SIAM J. Disc. Math, Vol. 4, N4 (1991), 528--534.

\bibitem{West} D. West, Introduction to Graph Theory, Prentice-Hall, New
Jersey, 2001.

\bibitem{YangLi} F. Yang and X. Li, Interval coloring of $\left(3,4\right)
$-biregular bigraphs having two $\left(2,3\right)$-biregular
bipartite subgraphs, Appl. Math. Let. 24 (2011), 1574--1577.

\end{thebibliography}
\end{document}